\documentclass[a4paper,10pt]{article}

\usepackage{amsthm}
\usepackage{amsmath}
\usepackage{amsxtra}
\usepackage{amssymb}
\usepackage{thmtools}
\usepackage{mathrsfs}
\usepackage{undertilde}
\usepackage{turnstile}
\usepackage{comment}
\usepackage{enumitem}
\usepackage[colorlinks=true,linkcolor=blue]{hyperref}
\usepackage{url}
\usepackage[retainorgcmds]{IEEEtrantools}

\newcommand{\jensenurl}{https://www.mathematik.hu-berlin.de/~raesch/org/jensen.html}
\declaretheorem[name=Definition,style=definition,qed=$\dashv$,
numberwithin=section]{dfn}

\declaretheorem[name=Theorem,style=plain,sibling=dfn]{tm}
\declaretheorem[name=Theorem,style=plain,numbered=no]{tm*}
\declaretheorem[name=Fact,style=plain,sibling=dfn]{fact}
\declaretheorem[name=Fact,style=plain,numbered=no]{fact*}
\declaretheorem[name=Lemma,style=plain,sibling=dfn]{lem}

\declaretheorem[name=Corollary,style=plain,sibling=dfn]{cor}
\declaretheorem[name=Remark,style=definition,sibling=dfn]{rem}
\declaretheorem[name=Claim,style=plain]{clm}
\declaretheorem[name=Claim,style=plain,numbered=no]{clm*}

\declaretheorem[name=Sublaim,style=plain,numbered=no]{sclm*}

\newcommand{\cHull}{{\mathrm{c}\Hull}}
\newcommand{\nins}{\not\ins}
\newcommand{\pins}{\lhd}
\newcommand{\core}{\mathfrak{C}}
\newcommand{\ins}{\trianglelefteq}

\newcommand{\QQ}{\mathbb Q}

\newcommand{\PP}{\mathbb P}

\newcommand{\sub}{\subseteq}
\newcommand{\cross}{\times}
\newcommand{\all}{\forall}

\newcommand{\inter}{\cap}
\renewcommand{\int}{\inter}

\newcommand{\om}{\omega}
\newcommand{\pow}{\mathcal{P}}
\newcommand{\OR}{\mathrm{OR}}

\newcommand{\Hull}{\mathrm{Hull}}

\newcommand{\cut}{\backslash}

\newcommand{\Tt}{\mathcal{T}}

\newcommand{\Uu}{\mathcal{U}}

\newcommand{\Ll}{\mathcal{L}}

\newcommand{\rg}{\mathrm{rg}}

\newcommand{\crit}{\mathrm{cr}}

\newcommand{\rest}{\!\upharpoonright\!}
\newcommand{\com}{\circ}

\newcommand{\Ult}{\mathrm{Ult}}

\newcommand{\sats}{\models}
\newcommand{\elem}{\preccurlyeq}
\newcommand{\J}{\mathcal{J}}

\newcommand{\HOD}{\mathrm{HOD}}

\newcommand{\ZFC}{\mathsf{ZFC}}
\newcommand{\ZF}{\mathsf{ZF}}

\newcommand{\Coll}{\mathrm{Col}}

\newcommand{\id}{\mathrm{id}}

\newcommand{\forces}{\dststile{}{}}

\newcommand{\Vop}{\mathrm{Vop}}

\newcommand{\ZFR}{\mathrm{ZFR}}

\newcommand{\OD}{\mathrm{OD}}

\newcommand{\ot}{\mathrm{ot}}
\newcommand{\lpole}{\left\lfloor}
\newcommand{\rpole}{\right\rfloor}
\newcommand{\univ}[1]{\lpole #1\rpole}
\newcommand{\tu}{\textup}
\newcommand{\NN}{\mathbb{N}}

\newcommand{\trcl}{\mathrm{trcl}}

\DeclareMathOperator{\cof}{cof}
\DeclareMathOperator{\wfp}{wfp}
\DeclareMathOperator{\rank}{rank}

\begin{document}

\title{Mice with Woodin cardinals from a Reinhardt}

\author{Farmer Schlutzenberg\\
farmer.schlutzenberg@tuwien.ac.at}

\maketitle

\begin{abstract} 
If there is a Reinhardt cardinal then
\begin{enumerate}\item $M_n^\#(X)$
exists and is fully iterable (above $X$) for every transitive set $X$ and 
every $n<\om$; and
\item Projective Determinacy holds in every set generic extension.\end{enumerate}\end{abstract}

\section{Introduction}

The purpose of this note is to provide a proof
of the following result, which was originally announced in the author's abstract for the Oberwolfach set theory conference in 2020 \cite[p.~834, Theorem 7]{oberwolfach_2020},
and also in \cite[paragraph following Theorem 6.1]{con_lambda_plus_2}:\footnote{The proof presented here  uses a 2021 result of Gabe Goldberg
at one point, to get around a gap in the original proof.} 

\begin{tm}
 Assume $\ZF(j)$ and that there is a Reinhardt cardinal
 as witnessed by $j$. Then $M_n^\#(X)$ exists and is fully iterable \tu{(}above $X$\tu{)} for every transitive set $X$ and every $n<\om$.
\end{tm}

Recall here that $M_n(X)$ is the canonical proper class inner model containing $X$
and having $n$ Woodin cardinals $\delta>\rank(X)$,
and $M_n^\#(X)$ is its sharp. And the theory $\ZF(j)$ is basically ZF with an extra symbol $j$; it is described in detail below.

Well known results give the following corollary:

\begin{cor}
 Assume $\ZF(j)$ and that there is a Reinhardt cardinal
 as witnessed by $j$. Let $V[G]$ be any set generic extension of $V$. Then $V[G]$ satisfies Projective Determinacy.
\end{cor}

\begin{dfn}
The \emph{language of set theory with predicate}
$\Ll_{\dot{\in},\dot{j}}$ is the first order language with binary
predicate symbols
$\dot{\in}$ and $\dot{=}$ and predicate symbol $\dot{j}$.

The theory \emph{$\ZF(j)$} is the theory in 
$\Ll_{\dot{\in},\dot{j}}$
with all $\ZF$ axioms (with $\dot{\in}$ being membership and $\dot{=}$ equality), allowing all 
formulas of $\Ll_{\dot{\in},\dot{j}}$
for the Separation and Collection schemes.

The theory $\ZFR$ is the $\ZF(j)$ together 
 with the (single) axiom asserting that
\[ \text{``}\dot{j}:(V,\dot{\in})\to(V,\dot{\in})\text{ is 
}\Sigma_1\text{-elementary.''}\footnote{One could
add the axiom scheme, where for each 
$n\in\NN$, we have the assertion that $\dot{j}$
is $\Sigma_n$-elementary, but we will see that
this already follows.}\qedhere
\]
\end{dfn}

\subsection{Acknowledgments}

The author thanks the organizers of the Oberwolfach conference 2020, and for the opportunity to announce this and related work in the abstract   for the conference. (The actual conference was unfortunately cancelled.)

This work was funded by the Deutsche 
Forschungsgemeinschaft (DFG, German Research
Foundation) under Germany's Excellence Strategy EXC 2044-390685587,
Mathematics M\"unster: Dynamics-Geometry-Structure.

\section{Some background}\label{sec:Reinhardt_sharps}

\begin{dfn}
 Let $j:V_\delta\to V_\delta$ or $j:V\to V$ be $\Sigma_1$-elementary
 and $\kappa_0=\crit(j)$ and $\kappa_{n+1}=j(\kappa_n)$.
 Let $\lambda=\sup_{n<\om}\kappa_n$.
(Note that $\lambda\leq\delta$.)
 We write 
$\kappa_{n,j}=\kappa_n$ and $\lambda_j=\kappa_{\om,j}=\lambda$.
\end{dfn}

\begin{dfn}
Given a structure $M=(\univ{M},\in^M,\ldots)$ with universe $\univ{M}$,
and given $A\sub\univ{M}$, we say that $A$ is \emph{amenable} to $M$
iff for each $x\in\univ{M}$, we have
\[  
\{y\in\univ{M}\bigm|y\in A\text{ and }y\in^M x\}\in\univ{M}. \]

A \emph{class} of a model $M$ of $\ZF$ 
is a collection $X\sub M$ such that $X$ is amenable to $M$
and $(M,X)\sats\ZF$.
(Here we must be working in some background model which sees
$M$ and collections $X\sub M$, where we make this definition.)
\end{dfn}

\begin{dfn}\label{dfn:HOD_C(X)} Work in $\ZF$.
Let $C$ be a class. Then $\OD_C(X)$
denotes the class of all sets $y$ such that $y$
is definable from the predicate $C$
and parameters in $\OR\cup X$.
And $\HOD_C(X)$ denotes the class of all $y$
such that $\trcl(\{y\})\sub\OD_C(X)$.
And $\OD(X)=\OD_\emptyset(X)$ and $\HOD(X)=\HOD_\emptyset(X)$.
\end{dfn}

We give the relevant instance of Vopenka forcing:
\begin{lem}
 Assume $\ZF$. Then for every set $x$, $\HOD_{\{x\}}$
 is a set-generic extension of $\HOD$.
\end{lem}
\begin{proof}
 Let $\alpha\in\OR$ be such that $x\in V_\alpha$.
Let $\Vop^*$ be the partial order 
consisting of $\OD$
 subsets of $V_\alpha$, with $p\leq q\iff p\sub q$.
 Let $\Vop\in\HOD$ be the natural coding
 of $\Vop^*$ as a subset of some $\beta\in\OR$.
 Write $p\mapsto p^*$ for the natural bijection
 $\Vop^*\to\Vop$. Note that this is $\OD$.
 
 Let $G_x=\{p\in\Vop\bigm|x\in p^*\}$.
 Then by the usual Vopenka proof, $G_x$ is $(\HOD,\Vop)$-generic.
 We claim that $\HOD[G_x]=\HOD_{\{x\}}$.
 For clearly $\HOD[G_x]\sub\HOD_{\{x\}}$.
 So let $A\sub\eta\in\OR$ with $A\in\HOD_{\{x\}}$.
 Let $\varphi$ be a formula and $\xi\in\OR$ such that
 \[ A=\{\gamma<\eta\bigm|\varphi(x,\xi,\gamma)\}.\]
For $y\in V_\alpha$ let
 \[ A_y=\{\gamma<\eta\bigm|\varphi(y,\xi,\gamma)\}.\]
 Let $B_\gamma\in\Vop$ be the condition such that
 \[ B^*_\gamma=\{y\in V_\alpha\bigm|\gamma\in A_y\}=\{y\in 
V_\alpha\bigm|\varphi(y,\xi,\gamma)\},\]
which is an $\OD$ subset of $V_\alpha$.
Note that $\left<B_\gamma\right>_{\gamma<\eta}\in\HOD$.
Now note that
\[ \gamma\in A\iff B_\gamma\in G_x, \]
so $A\in\HOD[G_x]$, as desired.
\end{proof}

\begin{lem}Assume $(V,j)\sats\ZFR$. Then for no set $X$
 is $j\rest\OR$ amenable to $\HOD(X)$.
\end{lem}
\begin{proof}
Suppose otherwise. Then since $j$ is proper class
and $j\rest\alpha$ is a set of ordinals for each $\alpha\in\OR$,
there is $x\in X$ such that $j\rest\alpha\in\HOD_{\{x\}}$
for proper class many $\alpha\in\OR$. But then $j\rest\OR$
is amenable to $\HOD_{\{x\}}$. But by the previous lemma,
$\HOD_{\{x\}}=\HOD[G_x]$ is a set-generic extension of $\HOD$.
But $j\rest\OR$ determines
\[ j\rest\HOD:\HOD\to\HOD, \]
through the standard class wellordering of $\HOD$,
and (since each of the set segments of this are in $\HOD$)
it follows that $(\HOD[G_x],j\rest\HOD)\sats\ZFC$.
This contradicts
\cite{gen_kunen_incon}.\end{proof}

\begin{lem}\label{lem:j_rest_OR_non-amenable}Assume $(V,j)\sats\ZFR$ and let 
$(N_\om,j_\om)$ be the $\om$th iterate of $(V,j)$.
Then $j\rest\OR$ is not amenable to $N_\om[A]$ for any set $A$.
\end{lem}
\begin{proof} 
 We have $j\rest\lambda_j\notin N_\om$,
 since $\lambda_j$ is inaccessible in $N_\om$,
 but from $j\rest\lambda_j$ one can recover the critical sequence of $j$.
 So the following claim gives a contradiction, completing the proof:
 
\begin{clm} If there is some set $A$ such that $j\rest\OR$ is amenable to $N_\om[A]$,
then $j\rest\alpha\in N_\om$ for every $\alpha\in\OR$.\end{clm}
\begin{proof}
Suppose $j\rest\OR$ is amenable to $N_\om[A]$ for some set $A$.
By \cite[Theorem 3.15]{reinhardt_iterates} (or \cite[Lemma 3.9]{reinhardt_iterates} suffices here),
 we can fix a forcing $\PP\in N_\om$
 and an $(N_\om,\PP)$-generic $G$
 with $A\in N_\om[G]$. By our assumption,
 $j\rest\OR$ is amenable to $N_\om[G]$.
 
Fix $\alpha\in\OR$.
Let $\xi$ be the least ordinal such that in $V$,
$\xi$ is not the surjective
image of $\PP\cross\alpha$.
Then considering the $\PP$-forcing relation in $N_\om$
as in \cite[\S2]{reinhardt_iterates}, we can find 
$X\in N_\om$,
with $X\sub\xi$, such that $j\rest X\in N_\om$,
and $\eta=\ot(X)\geq\alpha$. Let $\pi:\eta\to X$
be the increasing enumeration of $X$, so $\pi\in N_\om$.
Now $j(N_\om)=N_\om$, so
\[ j\rest N_\om:N_\om\to N_\om, \]
so $j(\eta,X,\pi)=(j(\eta),j(X),j(\pi))\in N_\om$,
and $j(\pi):j(\eta)\to j(X)$ is the increasing enumeration
of $j(X)$. We have $j\com\pi=j(\pi)\com j\rest\eta$,
and $j\com\pi\in N_\om$ and $j(\pi)\in N_\om$, so $j\rest\eta\in N_\om$,
but $\eta\geq\alpha$, so we are done.
\end{proof}

With the claim, we have proven the lemma.
\end{proof}

\section{$M_n^\#(X)$}
The following is by \cite{iter_for_stacks}:

\begin{fact}
 Assume ZF. Let $n<\om$ and $X$ be a transitive set and suppose $M_n^\#(X)$
 exists and is $(0,\OR)$-iterable. Then $M_n^\#(X)$ is $(0,\OR,\OR)^*$-iterable.
\end{fact}

\begin{rem}
 Note that one doesn't need to assume any choice
 for the result above. There is a use of the existence of sufficient regular cardinals involved in the proof in \cite{iter_for_stacks}, but that proof can be executed in $L(A,X)$ for some class $A$ of ordinals,
 where we get such things.
\end{rem}

And a standard comparison argument shows:
\begin{fact} Assume ZF. Let $n<\om$
 and $X$ a transitive set and suppose that $M_n^{\#}(X)$
 exists and is $(0,\OR)$-iterable.
 Then there is a unique $(0,\OR)$-strategy for $M_n^{\#}(X)$; we denote this by $\Sigma_{M_n^{\#}(X)}$.
\end{fact}

\begin{lem}\label{lem:M_n^sharp_basic_closure}
 Assume ZF. Let $n<\om$, and suppose that for every transitive set $X$, $M_n^\#(X)$ exists and is $(0,\OR)$-iterable. Let $G$ be any set-generic filter over $V$. Then:
 \begin{enumerate}\item $V[G]\sats$``For every transitive set $X$, 
 $M_n^\#(X)$ exists  and is $(0,\OR)$-iterable'',
  \item For all transitive $X\in V$, we have:
  \begin{enumerate}[label=--]\item 
  $(M_{n}^{\#}(X))^{V[G]}=(M_{n}^{\#}(X))^V$;
  let $N=(M_{n}^{\#}(X))^V$, 
  and \item $\Sigma_{N}^V=\Sigma^{V[G]}_N\rest V$.
  \end{enumerate}
  \item For any transitive set $Y$ and class $C$:
  \begin{enumerate}
  \item $\HOD_C(Y)\sats$``For all transitive sets $X$, $M_n^\#(X)$ exists and is $(0,\OR)$-iterable'',
  \item  for all transitive sets $X\in\HOD_C(Y)$,
  we have:
  \begin{enumerate}\item $(M_{n}^{\#}(X))^{\HOD_C(Y)}=(M_{n}^{\#}(X))^V$; let $N=(M_n^{\#}(X))^V$,
  \item $\Sigma_{N}^{\HOD_C(X)}=\Sigma_{N}^V\rest\HOD_C(Y)$.
  \end{enumerate}
   \end{enumerate}
   \end{enumerate}
\end{lem}
\begin{proof}[Proof sketch]
This is just a standard kind of absoluteness of iterability argument, involving reflection;
see for example the arguments in \cite[\S2]{cmip}.
The basic point is much like in the proof of the following lemma, so one should see that.
The proof is by induction on $n$,
so we may assume that if $n>0$ then it already holds for $n-1$, and hence we already have closure under the $M_{n-1}^{\#}$-operator etc. Using this operator, we can define a putative strategy for a putative $M_{n}^{\#}(X)$, and then prove that it works. Compare with the proof below.

Actually, we give a more detailed instance of this proof in the proof of Claim \ref{clm:iterability} of the proof of Theorem \ref{tm:Mnsharp} later.
\end{proof}

In the following lemma we use robust constructions
as in \cite{Kwithoutmeas}.
\begin{lem}
 Assume ZF. Let $n<\om$, and suppose that for every transitive set $X$, $M_n^\#(X)$ exists and is $(0,\OR)$-iterable. Then either:
 Moreover, either:
 \begin{enumerate}
  \item For every transitive set, $M_{n+1}^{\#}(X)$ exists and is $(0,\OR)$-iterable, or
  \item There is a transitive set $X$ such that:
  \begin{enumerate}\item There is no $(0,\OR)$-iterable $M_{n+1}^{\#}(X)$,
  and moreover, for all $\alpha\in\OR$
  with $X\in V_\alpha$,
  there is no $(0,\OR)$-iterable $M_{n+1}^{\#}(V_\alpha)$.
  \item Let $C$ be any class\footnote{In the end we not actually need to quantify over classes here.}. Let $G$ be $(\HOD_C(X),\Coll(\om,X))$-generic.
  Work in $\HOD_C(X)[G]$ (where ZFC holds). There we have the following:
  \begin{enumerate}\item Let $\left<N_\alpha\right>_{\alpha\leq\theta}$ be a robust $K^c$ construction over $N_0=\J(X)$. Then for  every $\alpha\leq\theta$ and every $m<\om$,
  $\core_m(N_\alpha)$ is $(m,\OR)$-iterable,
  and $N_\alpha$ is $(n+1)$-small;
   moreover, there are no exotic creatures of the construction.
  \item 
  There is  a maximal such robust $K^c$ construction with $\theta=\OR$.
  \end{enumerate}
  \end{enumerate}
 \end{enumerate}
 \end{lem}
 \begin{proof}[Proof sketch]
  Suppose that for some $X$, there is no $(0,\OR)$-iterable $M_{n+1}^{\#}(X)$. Work in $\HOD_C(X)[G]$
  as described in the statement of the lemma,
  and fix a robust $K^c$ construction as there.
  Then roughly, we use the $M_n^\#$-operator
  in $\HOD_C(X)[G]$ to produce the Q-structures
  which guide the formation of iteration trees on $\core_m(N_\alpha)$. Because $M_n^\#(Y)$
  exists and is fully iterable for every $Y$
  (in $\HOD_C(X)[G]$), this is fairly reasonable.
  (That is, given a tree $\Tt$ of limit length,
  we look at initial segments of $M_n^\#(M(\Tt))$
  for the next Q-structure.)
  It gives an absolute enough definition that
  in case it fails, we can reflect it down by taking a countable substructure, and use robustness
  to ensure the existence of branches where needed, at the countable level;
  see \cite{jensen_robust}
  and \cite{Kwithoutmeas}. There is a little more detail
  to handle in case we are iterating a non-sound structure which has Woodin cardinals.
  (So, the actual Q-structure
  we need might fail to be $\delta(\Tt)$-sound,
  and so it could differ from $M_n^\#(M(\Tt))$.
  In this case we instead need to look for a $\Tt$-cofinal branch $b$ such that the $\delta(\Tt)$-core of $M^\Tt_b$ is a segment of $M_n^\#(M(\Tt))$.)
 \end{proof}

\section{ZFR and
 mice with some Woodin cardinals}
 
\begin{tm}\label{tm:Mnsharp}
$\ZFR$ implies that for every set $X$ and $n<\om$, $M_n^\#(X)$ exists
and is $(0,\OR)$-iterable.
\end{tm}

The iterability here refers to (fine) iteration trees, which are above $X$.

\begin{rem}The $n=0$ case of the theorem (that is, that every set has a sharp)
 was observed independently and earlier by Gabe Goldberg.
 
 The basic form of the proof to follow is like that of \cite[Theorem 7.2]{con_lambda_plus_2}; however,
 the proof to follow was actually found earlier than that one (excluding the detail arranging that $\Omega$ is regular in $V$).
\end{rem}

\begin{proof}
 The proof is by induction on $n<\om$; we start with a proof
 that every set has a sharp.
We adopt the notation from the 
previous section.
 It suffices to see that for cofinally many $\eta\in\OR$,
 $V_\eta$ has a sharp. So fix a limit $\eta$ such that $j(\eta)=\eta$.
 Then
 \[ j\rest L(V_\eta):L(V_\eta)\to L(V_\eta).\]
Let $E$ be the $V_\eta$-extender derived from $j$
(see \cite{ZF_extenders}).
Let $U=\Ult_0(L(V_\eta),E)$ (note that the subscript $0$
means that the ultrapower is formed with pairs $[a,f]$ where $a\in\left<V_\eta\right>^{<\om}$ and $f\in L(V_\eta)$).
The ultrapower satisfies \L o\'s's theorem, because it is derived from $j$.
That is, let $f:\left<V_\eta\right>^{<\om}\to L(V_\eta)$
with $f\in L(V_\eta)$ and let $a\in\left<V_\eta\right>^{<\om}$
and $\varphi$ be a formula and suppose that 
\[ \all^*_{E_a}k\ \big[L(V_\eta)\sats\exists x\varphi(x,f(k))\big].\]
Then there is $b\in\left<V_\eta\right>^{<\om}$
with $a\sub b$, and $g\in L(V_\eta)$, such that
\[ \all^*_{E_b}k\ \big[L(V_\eta)\sats\varphi(g(k),f^{ab}(k))\big].\]
(See \cite{ZF_extenders} for more details.)
So we get $U=L(V_\eta)$ and
get an elementary ultrapower map
\[ i=i^{L(V_\eta),0}_{E}:L(V_\eta)\to 
U=L(V_\eta), \]
and the factor map
$\ell:U=L(V_\eta)\to L(V_\eta)$, where 
$\ell([a,f]^{L(V_\eta)}_E)=j(f)(a)$;
by \L o\'s's theorem,
$\ell$ is well-defined and elementary.
Note
 $i\rest V_\eta=j\rest V_\eta$
and $\ell\com i=j\rest L(V_\eta)$. We have
$\ell\rest V_\eta\cup\{\eta\}=\id$.
So it suffices to see that $\ell\neq\id$.
But if $\ell=\id$ then
$i=j\rest L(V_\eta)$, so $j\rest\OR$ is definable
over $L(A)$ for some set $A$,
and hence is amenable to $N_\om[A]$,
contradicting Lemma \ref{lem:j_rest_OR_non-amenable}.
This completes the $n=0$ case.

 We next proceed through a (finite stage)
core model induction.\footnote{
The author does not know whether one can adapt
the core model theory of \cite{Kwithoutmeas} to $\ZF$ (or $\ZFR$);
if one can do that successfully, it might simplify
the arguments to follow. Instead of that, we apply
the standard core model theory directly in models of choice.} So fix $n<\om$ and 
suppose that for all sets $X$, 
$M_n^\#(X)$ exists and is $(0,\OR)$-iterable (of course, this means above $X$).
By Lemma \ref{lem:M_n^sharp_basic_closure},
for every transitive set $X$, $\HOD(X)$ is closed under the $M_n^\#$-operator,
and under the corresponding iteration strategies, and can define the operator
and strategies. Likewise $\HOD_A(X)$, for classes $A$. Now we
want to verify the theorem at $n+1$.

Fix $\eta\in\OR$ such that $j(\eta)=\eta$.
It suffices to see that $M_{n+1}^\#(V_\eta)$ exists (and is fully iterable).
Let $H=\HOD(V_\eta)$ and $H_j=\HOD_j(V_\eta)$
(Definition \ref{dfn:HOD_C(X)}).
Let $g$ 
be $(H_j,\Coll(\om,V_\eta))$-generic;
so $H[g]\sats\ZFC$ and $H_j[g]\sats\ZFC$. By standard arguments, 
$H[g]$ and $H_j[g]$ also satisfy
``for every set $X$, $M_n^\#(X)$ exists and is $\OR$-iterable'',
and agree with $H,H_j$
over the restrictions of 
these operators and their 
corresponding iteration 
strategies to $H,H_j$ respectively. (One could use \cite{iter_for_stacks} to 
extend the 
iteration 
strategies to the generic extensions, but it should be easier than this in the 
current 
context.)

\begin{clm}\label{clm:iterability} Suppose 
$H[g]\sats$``$M=M_{n+1}^\#(V')$ exists and is $(0,\OR)$-iterable''
where $V'=V_\eta$.
Then $M\in H$ and $M$ is  also $(0,\OR)$-iterable in $H$ and $V$.
\end{clm}
\begin{proof}
We get $M\in H$ by the uniqueness of $M$ in $H[g]$ and homogeneity of the 
collapse. And $M$ is similarly $(0,\OR)$-iterable in $H$.
One can use an absoluteness argument to see that the iteration strategy 
extends to $V$, using that $H$ is closed under the $M_n^\#$-operator,
and that every subset of $\OR\cross V_\eta$
in $V$ is set-generic over $H$. That is, we claim that
$M$ is $(0,\OR)$-iterable in $V$, via the following putative strategy $\Gamma$:
given a limit length $0$-maximal tree $\Tt$ on $M$ according to $\Gamma$,
then $\Gamma(\Tt)$ is the branch determined by the Q-structure
$Q\ins M_n^\#(M(\Tt))$ for $M(\Tt)$ (if there is such).
Suppose there is some $\Tt\in V$ according to $\Gamma$
for which this fails to yield a Q-structure 
$Q$, or such that $Q$ fails to yield  a wellfounded $\Tt$-cofinal branch,
or etc. Fix some $\xi\in\OR$ such that there is a forcing $\PP\in V_\xi^H$
and an $(H,\PP)$-generic $G$ such that the counterexample
$\Tt$, etc, appears in $V_\xi^H[G]$. We have $M_n^\#(V_\xi^H)\in H$,
and $M_n^\#(V_\xi^H)[G]$ is equivalent to $N=M_n^\#(V_\xi^H[G])$,
and $N$ computes $M_n^\#(Y)$ for all $Y\in V_\xi^H[G]$,
and thus can be used to verify the construction of $\Tt$ (and possibly $Q$) etc.
All of this gets forced by some $p\in\PP$ about some names in $V_\xi^H$.
Thus, in $H[g]$ we have the tree $T$ of attempts to build
a countable elementary substructure $X\elem M_n^\#(V_\xi^H)$,
containing the relevant objects $\PP$, etc, including all elements of $V_\eta$.
Letting $\bar{M}$ be the transitive collapse of $X$,
in $H[g]$ we can choose a generic for $\bar{\PP}$ (the collapse of $\PP$), etc,
but this easily contradicts the
iterability of $M_{n+1}^\#(V_\eta)$ in $H[g]$.
\end{proof}

\begin{clm}
 Suppose $H[g]\sats$``There is a sound premouse $N$ of the form of $M_{n+1}^{\#}(V')$ such that for every countable $\bar{N}$ and every elementary $\pi:\bar{N}\to N$, $\bar{N}$ is $(0,\om_1)$-iterable''.
 Then $H[g]\sats$``$N=M_{n+1}^{\#}(X)$ and $N$ is $(0,\OR)$-iterable''.
\end{clm}
\begin{proof}
 The proof is similar to the foregoing one;
 we use the totality of the $M_n^\#$-operator
 to define a putative iteration strategy for $N$,
 and show that it does indeed work, by taking a countable elementary substructure of any failure,
 and running a comparison argument, using
 the $(0,\om_1)$-iterability of the countable structure; if a comparison reaches stage $\om_1$
 then the Q-structure provided by the $(\om_1+1)$-iterable structure with which we are comparing,
 yields a branch through the tree on the substructure $\bar{N}$.
\end{proof}

So it suffices to see that $H[g]\sats$``There is a countably iterable 
$M_{n+1}^\#(V')$'' where $V'=V_\eta$, so suppose otherwise; we will reach a 
contradiction.

Work in $H[g]$. All premice, robust constructions, etc, in what follows,
are over $V'$ (which is countable in $H[g]$), so we drop the phrase ``over $V'$''.
By our assumption, every robust
construction is $(n+1)$-small,
and therefore does not reach
any $M_n^\#$-closed model satisfying ``there is a Woodin cardinal''.
(Here $W$ is \emph{$M_n^\#$-closed} if for every $R\pins W$,
we have $M_n^\#(R)\pins W$.)

Further, if $W$ is any countably iterable $M_n^\#$-closed premouse
then $W\sats$``there is no Woodin cardinal''
and $W$ is $(0,\OR)$-iterable,
via strategy guided by Q-structures
$Q\ins M_n^\#(M(\Tt))$ (that is, analogous to $\Gamma$ above).
Thus, we can use core model theory
relative to the $M_n^\#$-operator.

Work in $V$. Note that if $W\in H\sats$``$W$ is a countably iterable 
$M_n^\#$-closed premouse'' then $W$ is $(0,\OR)$-iterable,
via a strategy as above.
Note that
\[ j\rest H:H\to H.\]
Let $E$ be the $V_\eta$-extender derived from $j$.
Let $J=\Ult_0(H,E)$; then the ultrapower satisfies
Los' theorem and we get an elementary $i^{H,0}_E:H\to J$
and the natural factor map $k:J\to H$ with $k\com i^{H,0}_E=j\rest H$.
So $j\rest V_\eta\cup\{V_\eta\}\sub i^{H,0}_E$,
so $\crit(k)>\eta$, if $\crit(k)$ exists.
In fact, $\crit(k)$ exists, because otherwise
$i^{H,0}_E\rest\OR=j\rest\OR$, but $i^{H,0}_E$ is amenable
to $N_\om[G]$ for some $N_\om$-set-generic $G$,
contradicting Lemma \ref{lem:j_rest_OR_non-amenable}. So $\crit(k)>\eta$. Let 
$\mu=\crit(k)$.
Let $F$ be the $(\mu,k(\mu))$-extender over $J$ derived from
$k$ (so $F$ measures $\pow(\mu^{<\om}\cross V_\eta)$). 
Now $E,F$ can be added generically to $H$
via Vopenka forcings $\QQ_E,\QQ_F$. Let $\xi\in\OR$
be such that $\xi>k(\mu)$ and $\QQ_E,\QQ_F\in V_\xi^{H}$ and $j(\xi)=\xi$.

Let $\theta$ be a regular cardinal with $\om<\theta<\crit(j)$
(note $\crit(j)$ is an inaccessible limit of inaccessibles). 
Let $\tau_0\in\OR$ be such that:
\begin{enumerate}[label=--]
 \item $\tau_0>\xi$ and $\cof(\tau_0)=\theta$,
 \item $j``\tau_0\sub\tau_0$ (it now follows that $j(\tau_0)=\tau_0$),
\item there is no $(\alpha,\pi)$ such that $\alpha<\tau_0$
and $\pi:V_\alpha\to\tau_0$ is a surjection, and
\item for every $\alpha<\tau_0$
and $A\sub V_\eta\cross\alpha$,
there is a forcing $\PP\in V_{\tau_0}^H$ and an $(H,\PP)$-generic
filter $G$ such that $A\in H[G]$.
\end{enumerate}

Note that $H\sub H_j$ and
${H_j}[g]\sats\ZFC$.
Let $\tau=(\tau_0^+)^{{H_j}[g]}$.
Note that
\begin{equation}\label{eqn:tau} 
H_j\sats\text{``}\tau=\text{least }\chi\in\OR\text{  s.t. }
\neg\exists\text{ surjection }\sigma:V_\eta\cross\tau_0\to 
\chi\text{''.}\end{equation}
By $\ZFC$, $\tau$ is regular in ${H_j}[g]$, hence also in 
${H_j},H,H[g]$. Note that 
${H_j}[g]\sats$``$\tau_0^{\aleph_0}=\tau_0$'',
so $H_j[g]\sats$``$\alpha^{\aleph_0}<\tau$ for each $\alpha<\tau$''; it follows that $H[g]$ also satisfies these two statements (but it seems that we might have $(\tau_0^+)^{H[g]}<\tau$).
Also,
\begin{equation}\label{eqn:tau_strongly_reg}{H_j}[g]\sats
\neg\exists (\alpha,\sigma)\ [\alpha<\tau\text{ and  
}\sigma:V_\eta^{<\om}\cross\alpha^{<\om}\to\tau\text{ is 
cofinal}];\end{equation}
It follows
that $H,H[g]$ agree about which ordinals have cofinality $\tau$.

\begin{clm} $j(\tau)=\tau$.\end{clm}
\begin{proof} $j(j)$ is definable (over $V$)
from $j$, so
${H_{j(j)}}\sub{H_j}$.
Note
\[ j\rest{H_j}:{H_j}\to{H_{j(j)}} \]
is elementary. So lifting line (\ref{eqn:tau}) with $j$,
\begin{equation}\label{eqn:j(tau)} 
{H_{j(j)}}\sats\text{``}j(\tau)=\text{least }\chi\in\OR\text{  s.t. }
\neg\exists\text{ surjection }\sigma:V_\eta\cross\tau_0\to 
\chi\text{'',}\end{equation}
and since ${H_{j(j)}}\sub{H_j}$,
therefore $j(\tau)\leq\tau$,  so 
$j(\tau)=\tau$.
\end{proof}

By \cite[Theorem 3.14]{goldberg_measurable_cardinals},
there is a club class of cardinals $\varepsilon$
such that either $\varepsilon$ or $\varepsilon^+$ is measurable. So by intersecting with more clubs,
let $\Omega_0>\tau$ be a cardinal with properties like $\tau_0$, and either $\Omega_0$ or $\Omega_0^+$ is measurable. Let $\Omega\in\{\Omega_0,\Omega_0^+\}$ be measurable; in particular, $\Omega$ is regular (in $V$). Note that both $H[g]$ and  $H_j[g]$ satisfy
``$\tau,\Omega$ are regular, $2^{<\tau}<\Omega$, and $\all\alpha<\Omega\ [\alpha^\om<\Omega]$''.
Therefore we meet the requirements for developing $K$
in these models, as described in \cite[p.~6]{Kwithoutmeas}.

Work in $H[g]$. We follow \cite{Kwithoutmeas},
using notation as there. 
Let $W$, $W^*$ be as there, with $W^*=W$ in
Cases 1 and 2 of \cite[pp.~6,7]{Kwithoutmeas}.
Likewise, let $S(W^*)$ be as there;
so if $W$ 
is a mini-universe
then $W^*=W$ and $S(W^*)=S(W)$ is the stack over $W$, and otherwise 
$S(W^*)=W^*$.
Note that $W,W^*,S(W^*)$ are defined in $H$ and in $V$ from the parameters 
$V_\eta,\Omega,\tau$
(by homogeneity of the collapse and as bicephalus arguments give uniqueness of 
next 
extenders).
As discussed above and by \cite{jensen_robust} and \cite{Kwithoutmeas}, $W^*$ is 
fully iterable, in $H$, $H[g]$ and $V$,
and $H[g]\sats$``$W^*$ is stably-universal''
(universal with respect to stable weasels),
which implies $W^*$ is $M_n^\#$-closed.

Let $K=\widetilde{K}(\tau,\Omega)^{H[g]}$
and $\pi:K\to S(W^*)$ the uncollapse map,
which is elementary. 
So $K,\pi$ are also defined from the parameters $V_\eta,\tau,\Omega$
in $H[g]$, in $H$ and in $V$.
The proofs of \cite[Lemma 4.27, Lemma 4.31]{Kwithoutmeas}
go through as there, and hence 
$\tau\sub K$. Let $\tau_0'=\sup\pi``\tau_0$.

\begin{clm}$\tau_0'<\Omega$.
\end{clm}
\begin{proof}Suppose otherwise. Since
$\Omega$ is regular in $H[g]$,
therefore $\tau_0'>\Omega$. Therefore $W=W^*$ is a mini-universe
and $S(W)$ has largest cardinal $\Omega$,
so $\Omega\in\rg(\pi)$, and letting $\pi(\bar{\Omega})=\Omega$,
therefore $\bar{\Omega}<\tau_0$ and $\bar{\Omega}$
is the largest cardinal of $K$. But $\tau\sub K$
and $\tau_0$ is a cardinal in $H[g]$, hence also a cardinal in $K$,
a contradiction. 
\end{proof}

Working in $V$, from the parameters $\eta,\tau,\Omega$,
we can define a $\tau_0$-very soundness witness $Y\in H$;
that is, a stably-universal weasel $Y$ with the $\alpha$-definability
property at all $\alpha<\tau_0$.\footnote{Recall that
our premouse language
has symbols for all elements in $V_\eta$,
so in particular, this is trivial for $\alpha<\eta$.}
For note first that for each $\Gamma\in H[g]$, if $H[g]\sats$``$\Gamma$ is 
$W^*$-thick'',
then forcing calculations give some $\Gamma'\sub\Gamma$ such that 
$\Gamma'\in H$
and $H\sats\Coll(\om,V_\eta)\forces$``$\check{\Gamma'}$ is 
$\check{W^*}$-thick''.
So for each $\alpha\in\tau_0'\cut\rg(\pi)$,
there is $\Gamma\in H$ such that 
$H\sats\Coll(\om,V_\eta)\forces$``$\check{\Gamma}$ is $\check{W^*}$-thick''
and $\alpha\notin\Hull^{S(W^*)}(\Gamma)$,
and we can find some $x\in V_\eta$ and ordinals $\beta_\alpha<\xi_\alpha$
such that some such $\Gamma$ is definable over
$V_\xi$ from $(x,\eta,\beta_\alpha)$. Letting 
$\xi$ be such that $\xi\geq\sup_{\alpha<\tau_0'}\xi_\alpha$
and $V_\xi\elem_{10}V$,
we may assume $\xi_\alpha=\xi$ for all $\alpha$,
and then may assume $\beta_\alpha=\alpha$ for all $\alpha$,
by minimizing other ordinals. Therefore
in $V$, from parameters $(\tau,\Omega,\eta)$,
we can define the function sending $(x,\alpha)$
to the minimal choice $\Gamma_{x\alpha}$ for $\alpha$,
if there is one determined by $x\in V_\eta$. This function
is in $H$. Let $\Gamma$ be the intersection of its range.
Note that $H\sats\Coll(\om,V_\eta)\forces$``$\check{\Gamma}$ is 
$\check{W^*}$-thick'' (since $\tau_0'<\Omega$).
Let $Y'=\Hull^{S(W^*)}(\Gamma)$, $Y^+$ be the transitive collapse of $Y'$,
and $Y=Y^+|\Omega$; then $Y$ works.)

Since $Y$ was defined from $(\eta,\tau,\Omega)$,
we have $j\rest Y:Y\to Y$. Now let
$Y_1=\Ult_0(Y,E)$. We have \L o\'s's Theorem
and $i=i^{Y,0}_E:Y\to Y_1$ is elementary and the natural factor map $k_Y:Y_1\to Y$
is elementary with $k_Y\com i=j\rest Y$
and $\crit(k_Y)>\eta$, if $\crit(k_Y)$ exists. Note that in fact, $\crit(k_Y)$ exists and
(define $\kappa$ by)
\[\kappa=
\crit(k_Y)\leq \crit(k)=\mu<k(\mu)<\xi<\tau_0<\OR^Y.\]
 Let $F_Y$ be the $(\kappa,k_Y(\kappa))$-extender
over $Y_1$ derived from $k_Y$. Since $j(\tau_0)=\tau_0$,
we have $i^{Y,0}_E(\tau_0)=k_Y(\tau_0)=\tau_0$, so
$k_Y(\kappa)<\tau_0$, so our choice of $\tau_0$ ensures
that $(E,F_Y)$
can be added generically to $H$ via a forcing in $V_{\tau_0}^H$.
We can assume that $g$ is $H[E,F_Y]$-generic,
so $(E,F_Y)$
is added via  forcing over $H[g]$
via a forcing in $V_{\tau_0}^{H[g]}$,
so $H[g]$ and $H[g,E,F_Y]$ agree about
the relevant core model calculations (their collections of thick sets are similar enough).

Let $Y_2=\Ult_0(Y_1,F_Y)$
and $\ell_Y:Y_2\to Y$ the factor map; again we have \L o\'s
and elementarity and $\ell_Y\com i^{Y_1,0}_{F_Y}=k_Y$ and
$\crit(\ell_Y)>k_Y(\kappa)$.
Write $S_Y=S(Y)^{H[E,F_Y,g]}=S(Y)^{H[g]}$, etc;
recall that by convention, $S(Y)=Y$ unless $Y$ is a mini-universe.

\begin{clm}\label{clm:Ult_pres_stack}
Work in $H[E,F_Y,g]$. Then
\[ \Ult_0(S_Y,E)=S_{Y_1}\text{ and } \Ult_0(S_{Y_1},F_Y)=S_{Y_2} \]
and 
the ultrapower maps extend correspondingly, in that
\[ i^{Y,0}_E=i^{S_Y,0}_E\rest Y\text{ and }
i^{Y_1,0}_{F_Y}=i^{S_{Y_1},0}_{F_Y}\rest Y_1.\]
Moreover, in $V$, letting $k_Y^+:\Ult_0(S_Y,E)\to j(S_Y)$
be the factor map, we have $k_Y\sub k_Y^+$,
and likewise $\ell_Y\sub\ell^+_Y$.
\end{clm}
\begin{proof}
Assume that $W$ (and hence also $Y$) is a mini-universe,
as otherwise $S_Y=Y$, etc, and everything is trivial.

Since $\Omega$ is regular, etc, we have
\[ \Ult_0(Y,E)=\Ult_0(S_Y,E)|\Omega\text{ and }
i^{Y,0}_E=i^{S_Y,0}_E\rest Y.\] 
Also letting $k_Y^+:\Ult_0(S_Y,E)\to j(S_Y)$
be the factor map, we have $k_Y=k_Y^+\rest Y_1$.

In $V$, $\Ult_0(S_Y,E)$ is iterable, since we have $k_Y^+$. Hence it is also 
iterable in $H[E,F_Y]$ and $H[E,F_Y,g]$.
So $\Ult_0(S_Y,E)\ins S_{Y_1}$.

Now suppose  $\Ult_0(S_Y,E)\pins S_{Y_1}$.
Let $R\pins S_{Y_1}$ be least projecting to $\Omega$
with $R\nins\Ult_0(S_Y,E)$. Note that $R$ is $(n+1)$-small
(otherwise it would have $\rho_1=V_\eta$).

Note that $j(S_Y)=S_Y$, so 
$j\rest S_Y:S_Y\to S_Y$ and $\OR^{S_Y}=\OR^{\Ult_0(S_Y,E)}=\OR^{j(S_Y)}$,
so $\OR^R>\OR^{S_Y}$. 
Recall $k_Y^+:\Ult_0(S_Y,E)\to j(S_Y)=S_Y$ is the factor,
and $k_Y^+\com i^{S_Y,0}_E=j$.  Let $F_Y^+$ be the long $\Ult_0(S_Y,E)$-extender
derived from $k_Y^+$. So $\Ult_0(\Ult_0(S_Y,E),F_Y^+)=S_Y$
and $k_Y^+=i^{\Ult_0(S_Y,E),0`}_{F_Y^+}$.
Let $R'=\Ult_\om(R,F_Y^+)$ (that is, use all functions
definable from parameters over $R$ in forming the ultrapower)
and $\sigma_{RR'}:R\to R'$ the ultrapower map.
So $R'$ has the first-order theory of an iterable sound premouse,
with $\Omega+1\in\wfp(R')$ and $\rho_\om^{R'}=\Omega$.
In fact, $R'$ is wellfounded. For otherwise,
by the regularity of $\Omega$, there is $\bar{\Omega}<\Omega$
such that $R_0=\cHull_\om^{R'}(V_\eta\cup\bar{\Omega})$ is also illfounded,
but by the first-order properties, we can find such an $\bar{\Omega}$
with $R_0\pins R'$, and hence $\OR^{R_0}<\Omega$ wellfounded, a contradiction.

We have $R'\in{H_j}$. We claim that $H_j\models$``For every countable sound premouse $\bar{R}$, if there is an elementary $\pi:\bar{R}\to  R'$
then $\bar{R}$ is $(\om,\om_1+1)$-iterable''.
For working in $H_j$, where $\Omega$ is  regular, given any such $\Omega,\bar{R}$,
an appropriate application of condensation
(which $R'$ satisfies) gives that we can find $R''\pins Y$ and an elementary $\pi'':\bar{R}\to R''$,
and since $Y$ is iterable in $H_j$, this suffices.
But then by Claim \ref{clm:iterability}, in fact
$H_j\models$``$R'$ is $(\om,\OR)$-iterable''.

Now we claim that $R'\in H$. 
For $R'\in H[G^*]$ for $G^*$ for $H$-generic
filter $G^*$ for the appropriate instance of Vopenka forcing $\mathrm{Vop}$. So let $\gamma=\OR^{H'}$.
Suppose $R'\notin H$.
Then there is some $p\in\mathrm{Vop}$
and some $\mathrm{Vop}$-name $\dot{R}\in H$
such that $H\sats$``$p$ forces that $\dot{R}$ is a sound premouse satisfying condensation, $\check{S_Y}\pins\dot{R}$, and $\rho_\om^{\dot{R}}=\Omega$''.
A standard feature of Vopenka forcing is that for every $q\in\mathrm{Vop}$
there is a generic filter $G''\in V$ with $q\in G''$.
So fix such a filter $G''$ with $p\in G''$.
Let $R''=\dot{R}_{G''}$.
Then $\OR^{R'}=\OR^{R''}$ and $R',R''$ are both sound premice satisfying condenstion, with $\rho_\om^{R'}=\rho_\om^{R''}$. But $\Omega$ is regular (in $V$!),
and so Jensen's condensation argument (see the proof of  \cite[Lemma 3.1]{stacking_mice})  shows that $R'=R''$.
It follows that $R'\in H$.

Since $R'\in H$ and is iterable in $H_j$,
it is also iterable in $H$, and therefore $R'\pins S_Y$, a contradiction.

Essentially the same proof works for $S_{Y_2}$.
\end{proof}

\begin{clm}\label{clm:Y_1_def_prop}In $H[g,E,F_Y]$,
$Y_1$ has the $\alpha$-definability
property at all $\alpha<\tau_0$. Therefore 
$Y_1|\tau_0=Y|\tau_0=K|\tau_0$.\end{clm}
\begin{proof}We work in $H[g,E,F_Y]$.
By Claim \ref{clm:Ult_pres_stack}, $S_{Y_1}=\Ult_0(S_Y,E)$.
So let $\Gamma$ be any $Y_1$-thick
set. Then there is a $Y$-thick, $Y_1$-thick $\Gamma'\sub\Gamma$ such that
$i^{S_Y,0}_E\rest\Gamma'=\id$. We have
\[ V_\eta\cup\tau_0\sub\Hull^{S_Y}(V_\eta\cup\{V_\eta\}\cup\Gamma'),\]
\[ 
i^{S_Y,0}_E``(V_\eta\cup\tau_0)\sub\Hull^{S_{Y_1}}(i^{Y,0}_E``V_\eta\cup\{V_\eta\}
\cup\Gamma'),\]
\[ 
V_\eta\cup i^{S_Y,0}_E``\tau_0\sub\Hull^{S_{Y_1}}(V_\eta\cup\{V_\eta\}
\cup\Gamma'),\]
but since $S_{Y_1}=\Ult_0(S_Y,E)$ and the generators of $E$
are all in $V_\eta$, and also $j(\tau_0)=\tau_0$ and 
$i^{S_Y,0}_{E}(\tau_0)=\tau_0$, we have
\[ \tau_0=\sup i^{S_Y,0}_E``\tau_0\sub\Hull^{S_{Y_1}}(V_\eta\cup\{V_\eta\}\cup 
i^{S_Y,0}_E``\tau_0),\]
and therefore
\[ \tau_0\sub\Hull^{S_{Y_1}}(V_\eta\cup\{V_\eta\}\cup 
\Gamma'),\]
as desired.

The fact that $Y_1|\tau_0=Y|\tau_0=K|\tau_0$ is then a standard conclusion,
via comparing $Y_1$ with $Y$ and using that both $Y$ and $Y_1$
have the definability property at all $\alpha<\tau_0$.
\end{proof}

\begin{clm}\label{clm:phalanx_it}
 The phalanx $((Y,{<k_Y(\kappa)}),Y_2)$ is iterable in $V$,
 hence also in $H[E,F_Y]$ and $H[E,F_Y,g]$.
\end{clm}
\begin{proof}
In $V$, we can lift trees on $((Y,{<k_Y(\kappa)}),Y_2)$
to trees on $((Y,{<k_Y(\kappa)}),Y)$
via lifting maps $(\id,\ell_Y)$. This works because $\crit(\ell_Y)>k_Y(\kappa)$.
\end{proof}

\begin{clm}\label{clm:agmt_thru_k_Y(kappa)}$k_Y(\kappa)<\tau_0$ and $k_Y(\kappa)$ is an inaccessible cardinal
of each of $Y,Y_1,Y_2$, and $Y|k_Y(\kappa)=Y_1|k_Y(\kappa)=Y_2|k_Y(\kappa)$.
\end{clm}
\begin{proof}
Recall that $k_Y:Y_1\to Y$ is elementary
and $\kappa=\crit(k_Y)\leq\crit(k)=\mu<k(\mu)<\tau_0$,
and also $k_Y(\kappa)<\tau_0$ since $j(\tau_0)=\tau_0$.
So $\kappa=\crit(k_Y)$ is regular in
 $Y_1$ and 
 $k_Y(\kappa)$ is regular in $Y$.
We have $Y|\tau_0=Y_1|\tau_0$ by Claim \ref{clm:Y_1_def_prop},
and $\tau_0$ is a cardinal 
of $V$,
hence of $Y$ and $Y_1$.
And $k_Y(\kappa)<\tau_0$, 
so $k_Y(\kappa)$ is also a regular cardinal of $Y_1$,
and since $Y|\tau_0=Y|\tau_1$, not a successor, so is inaccessible there.
We also have 
$k_Y(\kappa)=i^{Y_1,0}_{F_Y}(\kappa)$
is inaccessible in $Y_2$, and $\ell_Y:Y_2\to Y$
is elementary with $k_Y(\kappa)<\crit(\ell_Y)$ (if $\crit(\ell_Y)$ exists),
so $Y_2|k_Y(\kappa)=Y|k_Y(\kappa)=Y_1|k_Y(\kappa)$.
\end{proof}

We can now complete the proof. We work in $H[E,F_Y,g]$.
We compare the phalanx mentioned in Claim \ref{clm:phalanx_it} 
with $Y_1$. We get a successful comparison $(\Tt,\Uu)$
with $\Tt$ on the phalanx and $\Uu$ on $Y_1$.
By the arguments in \cite{Kwithoutmeas}, $M^\Tt_\infty=M^\Uu_\infty$ and $b^\Tt,b^\Uu$ are non-dropping.
Both $Y$ and $Y_1$ have the definability property
at all $\alpha<\tau_0$. But $k_Y(\kappa)<\tau_0$.
So if $b^\Tt$ is above $Y$, the usual calculations with the definability
and hull properties give a contradiction.
So $b^\Tt$ is above $Y_2$. Let $Z=M^\Tt_\infty=M^\Uu_\infty$.
Then as in \cite{Kwithoutmeas}, $S_Z=M^{\Tt^+}_\infty=M^{\Uu^+}_\infty$
where $\Tt^+,\Uu^+$ are $\Tt,\Uu$ construed as trees on 
$((S_Y,{<k_Y(\kappa)}),S_{Y_2})$ and $S_{Y_1}$.
Let $\Gamma$ be $Y_2$-thick, $Y_1$-thick and $Z$-thick
and consist of fixed points for the embeddings $i^\Tt,i^\Uu,i^{S_{Y_1},0}_{F_Y}$.
Then since by Claim \ref{clm:Y_1_def_prop} we have
\[ \tau_0\sub\Hull^{S_{Y_1}}(V_\eta\cup\{V_\eta\}\cup\Gamma), \]
we get
\[ i^\Tt\com i^{S_{Y_1},0}_{F^Y}\rest\tau_0=i^\Uu\rest\tau_0.\]
But $\crit(i^\Tt)>k_Y(\kappa)$, and therefore $\crit(i^\Uu)=\kappa$ and
$i^\Uu(\kappa)=k_Y(\kappa)$
and $F_Y$ is the $(\kappa,i^\Uu(\kappa))$-extender
derived from $i^\Uu$. But by Claim \ref{clm:agmt_thru_k_Y(kappa)},
the comparison uses only extenders with index ${>i^\Uu(\kappa)}$
and $i^\Uu(\kappa)$ is a cardinal of $Y_1$, but then by the ISC,
the first extender used along $b^\Uu$ witnesses that $\kappa$ is superstrong in 
$Y_1$, a contradiction.
\end{proof}

\bibliographystyle{plain}
\bibliography{../bibliography/bibliography}

\end{document}